\documentclass{article}

\usepackage[overload]{abraces}
\usepackage{amsmath}
\usepackage{amssymb}
\usepackage{amsthm}
\usepackage{array,multirow}
\usepackage{enumerate}
\usepackage{tikz}
\usepackage{type1cm}
\usepackage{url}

\usepackage{pgfplots}
\usepgfplotslibrary{dateplot}
\pgfplotsset{compat=newest}
\usepgfplotslibrary{fillbetween}

\newtheorem{theorem}{Theorem}[section]
\newtheorem{corollary}[theorem]{Corollary}
\newtheorem{definition}[theorem]{Definition}
\newtheorem{example}[theorem]{Example}
\newtheorem{lemma}[theorem]{Lemma}
\newtheorem{proposition}[theorem]{Proposition}
\newtheorem{remark}[theorem]{Remark}

\makeatletter
\newcommand{\leqnomode}{\tagsleft@true\let\veqno\@@leqno}
\newcommand{\reqnomode}{\tagsleft@false\let\veqno\@@eqno}
\makeatother

\title{Bilevel Optimization under Uncertainty}

\author{J. Burtscheidt \footnotemark[1]\ \and M. Claus \footnotemark[1]}

\begin{document}

\maketitle

\renewcommand{\thefootnote}{\fnsymbol{footnote}}
\footnotetext[1]{Faculty of Mathematics, University of Duisburg-Essen, Campus Essen, Thea-Leymann-Stra\ss e 9,
D-45127 Essen, Germany, \url{[johanna.burtscheidt][matthias.claus]@uni-due.de}}
\renewcommand{\thefootnote}{\arabic{footnote}}

\begin{abstract}
    We consider bilevel linear problems, where the right-hand side of the lower level problems is stochastic. The leader has to decide in a here-and-now fashion, while the follower has complete information. In this setting, the leader's outcome can be modeled by a random variable, which gives rise to a broad spectrum of models involving coherent or convex risk measures and stochastic dominance constraints. We outline Lipschitzian properties, conditions for existence and optimality, as well as stability results. Moreover, for finite discrete distributions, we discuss the special structure of equivalent deterministic bilevel programs and its potential use to mitigate the curse of dimensionality.
\end{abstract}

\section{Introduction} In this chapter we consider bilevel optimization models with uncertain parameters. Such models can be classified based on the chronology of decision and observation as well as the nature of the uncertainty involved. A \textbf{bilevel stochastic program} arises, if the uncertain parameter is realization of some random vector with known distribution, that can only be observed once the leader has submitted their decision. In contrast, the follower decides under complete information.

\medskip

If upper and lower level objectives coincide, the bilevel stochastic program collapses to a classical \textbf{stochastic optimization problem with recourse} (cf. \cite[Chap. 2]{ShapiroDentchevaRuszczynski2014}). Relations to other mathematical programming problems are explored in the seminal work \cite{PatrikssonWynter1999} that also established the existence of solutions, Lipschitzian properties and directional differentiability of a risk-neutral formulation of a bilevel stochastic nonlinear model. Moreover, gradient descent and penalization methods were investigated to tackle discretely distributed stochastic mathematical programs with equilibrium constraints (SMPECs).

\medskip

Reference \cite{ChristiansenPatrikssonWynter2001} studies an application to topology optimization problems in structural mechanics. Many other applications are motivated by network related problems that inherit a natural order of successive decision making under uncertainty. Notable examples arise in telecommunications (cf. \cite{Werner2005}), grid-based (energy) markets (cf. \cite{CarrionArroyoConejo2009}, \cite{DempeKalashnikovPerezValdesKalashnykova2011}, \cite{KosuchLeBodicLeungLisser2012}, \cite{KovacevicPflug2013}) or transportation science (cf. \cite{ChenKimZhouChootinan2007}, \cite{Patriksson2008}). An extensive survey on bilevel stochastic programming literature is provided in \cite[Chap. 1.4]{Henkel2014}. 

\medskip

In \textbf{two-stage stochastic bilevel programming} leader and follower take two decisions: The decision on the respective first-stage variables is made in a \textbf{here-and-now} fashion, i.e. without knowledge of the realization of the random parameter. In contrast, the respective second-stage decisions are made in a \textbf{wait-and-see} manner, i.e. after observing the parameter (cf. \cite{AlizadehMarcotteSavard2013}).

\medskip

This chapter is organized as follows: In Sects. 2.1 to 2.5, we outline structural properties, existence and optimality conditions as well as stability results for bilevel stochastic linear problems while paying special attention to the modelling of risk-aversion via coherent/convex risk measures or stochastic dominance constraints. Sections 2.6 and 2.7 are devoted to the algorithmic treatment of bilevel stochastic linear problems, where the underlying distribution is finite discrete. An application of two-stage stochastic bilevel programming in the context of network pricing is discussed in Sect. 3. The chapter concludes with an overview of potential challenges for future research.

\section{Bilevel Stochastic Linear Optimization} \label{SecBilevelStochasticLinearOptimization}

While the analysis in this section is confined to the bilevel stochastic linear problems with random right-hand side, the concepts and underlying principles and can be easily transferred to stochastic extensions of more complex bilevel programming models. 

\subsection{Preliminaries} \label{SubSecPreliminaries}

We consider the optimistic formulation of a parametric bilevel linear program
\begin{equation}
\label{ParametricBLP}
\tag*{P($z$)}
\min_x \Big\{c^\top x + \min_y \{q^\top y \; | \; y \in \Psi(x,z)\} \; | \; x \in X \Big\},
\end{equation}
where $X \subseteq \mathbb{R}^n$ is a nonempty polyhedron, $c \in \mathbb{R}^n$ and $q \in \mathbb{R}^m$ are vectors, $z \in \mathbb{R}^s$ is a parameter, and the lower level optimal solution set mapping $\Psi: \mathbb{R}^n \times \mathbb{R}^s \rightrightarrows \mathbb{R}^m$ is given by
$$
\Psi(x,z) := \underset{y}{\mathrm{Argmin}} \; \{d^\top y \; | \; Ay \leq Tx + z\}
$$
with matrices $A \in \mathbb{R}^{s \times m}$, $T \in \mathbb{R}^{s \times n}$ and a vector $d \in \mathbb{R}^m$. 
Let $f: \mathbb{R}^n \times \mathbb{R}^s \to \mathbb{R} \cup \lbrace \pm \infty \rbrace$ denote the mapping
$$
f(x,z) := c^\top x + \min_y \{q^\top y \; | \; y \in \Psi(x,z)\}.
$$

\begin{lemma} \label{Lemmaf}
Assume $\mathrm{dom} \; f \neq \emptyset$, then $f$ is real-valued and Lipschitz continuous on the polyhedron $P = \{(x,z) \in \mathbb{R}^n \times \mathbb{R}^s \; | \; \exists y \in \mathbb{R}^m: Ay \leq Tx + z\}$.
\end{lemma}

\begin{proof}
By \cite{Eaves1971}, $\emptyset \neq \mathrm{dom} \; f \subseteq \mathrm{dom} \; \Psi$ implies $\mathrm{dom} \; \Psi = P$. Consequently, the linear program in the definition of $f(x,z)$ is solvable for any $(x,z) \in P$ by parametric linear programming theory (see \cite{Beer1977}). Consider any $(x,z), (x',z') \in P$. Without loss of generality, assume that $f(x,z) \geq f(x',z')$ and let $y' \in \Psi(x',z')$ be such that $f(x',z') = c^\top x' + q^\top y'$. Following \cite{KlatteKummer1984} we obtain
\begin{align*}
|f(x,z)-f(x',z')| \; &= \; f(x,z)- c^\top x' - q^\top y' \; \leq \; c^\top x + q^\top y - c^\top x' - q^\top y' \\
&\leq \; \|c\| \|x-x'\| + \|q\| \|y-y'\|
\end{align*}
for any $y \in \Psi(x,z)$. Let $\mathbb{B}$ denote the Euclidean unit ball, then \cite[Theorem 4.2]{KlatteThiere1995} yields
$$
\Psi(x',z')\subseteq \Psi(x,z)+ \Lambda \|(x,z)-(x',z')\|\mathbb{B}
$$
and hence $|f(x,z)-f(x',z')| \; \leq \; (\|c\| + \Lambda \|q\|) \|(x,z)-(x',z')\|$.
\end{proof}

\begin{remark}
An alternate proof for Lemma~\ref{Lemmaf} is given in \cite[Theorem 1]{Ivanov2014}. However, the arguments above can be easily extended to lower level problems with convex quadratic objective function and linear constraints.
\end{remark}

Linear programming theory provides verifiable necessary and sufficient condition for $\mathrm{dom} \; f \neq \emptyset$:

\begin{lemma}
$\mathrm{dom} \; f \neq \emptyset$ holds if and only if there exists $(x,z) \in \mathbb{R}^n \times \mathbb{R}^s$ such that
\begin{enumerate}[a.]
\item $\{y \; | \;  Ay \leq Tx + z\}$ is nonempty,
\item there is some $u \in \mathbb{R}^s$ satisfying $A^\top u = d$ and $u \leq 0$, and
\item the function $y \mapsto q^\top y$ is bounded from below on $\Psi(x,z)$.
\end{enumerate}
Under these conditions,
$$
\min_{y} \lbrace q^\top y \; | \; y \in \Psi(x',z') \rbrace
$$
is attained for any $(x',z') \in P$.
\end{lemma}

\subsection{Bilevel Stochastic Linear Programming Models} \label{SubSecBilevelStochasticLinearModels}

A bilevel stochastic program arises if the parameter $z = Z(\omega)$ in \ref{ParametricBLP} is the realization of a known random vector $Z$ on some probability space $(\Omega, \mathcal{F}, \mathbb{P})$ and we assume the following chronology of decision and observation:
\begin{center}
leader decides $x$ \hspace{10pt} $\rightarrow$ \hspace{10pt} $z = Z(\omega)$ is revealed \hspace{10pt} $\rightarrow$ \hspace{10pt} follower decides $y$.
\end{center}
Throughout the analysis, we assume the stochasticity to be purely exogenous, i.e. the distribution of $Z$ to be independent of $x$. 

\medskip

Let $\mu_Z := \mathbb{P} \circ Z^{-1} \in \mathcal{P}(\mathbb{R}^s)$ denote the Borel probability measure induced by $Z$. We shall assume $\mathrm{dom} \; f \neq \emptyset$ and that the lower level problem is feasible for any leader's decision and any realization of the randomness, i.e.
$$
X \subseteq P_Z := \lbrace x \in \mathbb{R}^n \; | \; (x,z) \in P \; \forall z \in \mathrm{supp} \; \mu_Z \rbrace.
$$

In two-stage stochastic programming, a similar assumption is known as \textbf{relatively complete recourse} (cf. \cite[Sect. 2.1.3]{ShapiroDentchevaRuszczynski2014}). In this setting, each leader's decision $x \in X$ gives rise to a random variable $f(x,Z(\cdot))$. We thus may fix any mapping $\mathcal{R}: \mathcal{X} \to \mathbb{R}$, where $\mathcal{X}$ is a linear subspace of $L^0(\Omega, \mathcal{F}, \mathbb{P})$ that contains the constants and  satisfies
$$
\lbrace f(x,Z(\cdot)) \; | \; x \in X \rbrace \subseteq \mathcal{X},
$$
and consider the \textbf{bilevel stochastic program}
\begin{equation}
\label{RSBLP}
\min_x \left\{\mathcal{R}[f(x,Z(\cdot))] \; | \; x \in X \right\}.
\end{equation}
Under suitable moment or boundedness conditions on $Z$ the classical $L^p$-spaces $L^p(\Omega, \mathcal{F}, \mathbb{P})$ with $p \in [1, \infty]$ are natural choices for the domain $\mathcal{X}$ of $\mathcal{R}$. Set
$$
\mathcal{M}^p_s := \left\{ \mu \in \mathcal{P}(\mathbb{R}^s) \; | \; \int_{\mathbb{R}^s} \|z\|^p~\mu(dz) < \infty \right\}
$$
denote the set of Borel probability measures on $\mathbb{R}^s$ with finite moments of order $p \in [1,\infty)$ and set
$$
\mathcal{M}^\infty_s := \left\{ \mu \in \mathcal{P}(\mathbb{R}^s) \; | \; \mathrm{supp} \; \mu_Z \; \text{is bounded} \right\}.
$$

\begin{lemma} \label{LemmaF}
Assume $\mathrm{dom} \; f \neq \emptyset$ and $\mu_Z \in \mathcal{M}^p_s$ for some $p \in [1,\infty]$. Then the mapping $F: P_Z \to L^0(\Omega, \mathcal{F}, \mathbb{P})$ given by $F(x) := f(x,Z(\cdot))$ takes values in $L^p(\Omega, \mathcal{F}, \mathbb{P})$ and is Lipschitz continuous with respect to the $L^p$-norm.
\end{lemma}

\begin{proof}
We first consider the case that $p$ is finite. By $(0,0) \in P$ and Lemma~\ref{Lemmaf}, there exist a constant $L_f$ such that
\begin{align*}
\| F(x) \|_{L^p}^p &\leq 2^p |f(0,0)|^p + 2^p \int_{\mathbb{R}^s} |f(x,z) - f(0,0)|^p~\mu_Z(dz) \\
&\leq 2^p |f(0,0)|^p + 2^p L_f^p \|x\|^p +  2^p L_f^p \int_{\mathbb{R}^s} \|z\|^p~\mu_Z(dz) < \infty
\end{align*}
holds for any $x \in P_Z$. Furthermore, for any $x, x' \in P_Z$ we have
$$
\| F(x) - F(x') \|_{L^p} = \left( \int_{\mathbb{R}^s} |f(x,z) - f(x',z)|^p~\mu_Z(dz) \right)^{1/p} \leq L_f \|x - x' \|.
$$

For $p = \infty$, Lemma~\ref{Lemmaf} implies that for any fixed $x \in P_Z$, the mapping $f(x,\cdot)$ is continuous on $\mathrm{supp} \; \mu_Z$. Thus, $\mu_Z \in \mathcal{M}^\infty_s$ yields
$$
\|F(x)\|_{L^\infty} \leq \sup_{z \in \mathrm{supp} \; \mu_Z} |f(x,z)| < \infty.
$$
Moreover, for any $x, x' \in P_Z$ we have
$$
\| F(x) - F(x') \|_{L^\infty} \leq \sup_{z \in \mathrm{supp} \; \mu_Z} |f(x,z) - f(x',z)| \leq L_f \|x-x'\|.
$$
\end{proof}

The mapping $\mathcal{R}$ in \eqref{RSBLP} can be used to measure the risk associated with the random variable $F(x)$. 

\begin{definition}
A mapping $\mathcal{R}: \mathcal{X} \to {\color{black}\mathbb{R}}$ defined on some linear subspace $\mathcal{X}$ of $L^0(\Omega, \mathcal{F}, \mathbb{P})$ containing the constants is called a \textbf{convex risk measure} if the following conditions are fulfilled:
\begin{enumerate}[a.]
\item \textbf{(Convexity}) For any $Y_1, Y_2 \in \mathcal{X}$ and $\lambda \in [0,1]$ we have
$$
\mathcal{R}[\lambda Y_1 + (1-\lambda)Y_2] \leq \lambda \mathcal{R}[Y_1] + (1 - \lambda) \mathcal{R}[Y_2].
$$
\item (\textbf{Monotonicity}) $\mathcal{R}[Y_1] \leq \mathcal{R}[Y_2]$ for all $Y_1, Y_2 \in \mathcal{X}$ satisfying $Y_1 \leq Y_2$ with respect to the $\mathbb{P}$-almost sure partial order.
\item (\textbf{Translation equivariance}) $\mathcal{R}[Y + t] = \mathcal{R}[Y] + t$ for all $Y \in \mathcal{X}$ and $t \in \mathbb{R}$.
\end{enumerate}
A convex risk measure $\mathcal{R}$ is \textbf{coherent} if the following holds true:
\begin{enumerate}[d.]
\item (\textbf{Positive homogeneity}) $\mathcal{R}[t Y] = t \cdot \mathcal{R}[Y]$ for all $Y \in \mathcal{X}$ and $t \in [0,\infty)$.
\end{enumerate} 
\end{definition}

\begin{definition}
A mapping $\mathcal{R}: \mathcal{X} \to {\color{black}\mathbb{R}}$ is called \textbf{law-invariant} if for all $Y_1, Y_2 \in \mathcal{X}$ with $\mathbb{P} \circ Y_1^{-1} = \mathbb{P} \circ Y_2^{-1}$ we have $\mathcal{R}[Y_1] = \mathcal{R}[Y_2]$.
\end{definition}

Coherent risk measures have been introduced in \cite{ArtznerDelbaenEberHeath1999}, while the analysis of convex risk measures dates back to \cite{FoellmerSchied2002}. A thorough discussion of their analytical traits is provided in \cite{FoellmerSchied2011}. Below we list some risk measures that are commonly used in stochastic programming (cf. \cite[Sect. 6.3.2]{ShapiroDentchevaRuszczynski2014}).

\begin{example}{Examples}
a. The \textbf{expectation} $\mathbb{E}: L^1(\Omega, \mathcal{F}, \mathbb{P}) \to \mathbb{R}$,
$$
\mathbb{E}[Y] = \int_\Omega Y(\omega)~\mathbb{P}(d\omega)
$$
is a law-invariant and coherent risk measure that turns \eqref{RSBLP} into the \textbf{risk neutral bilevel stochastic program}
$$
\min_x \left\{\mathbb{E}[F(x)] \; | \; x \in X \right\}.
$$

\noindent
b. The \textbf{expected excess of order $p \in [1, \infty)$} over a predefined level $\eta \in \mathbb{R}$ is the mapping $\mathrm{EE}_{\eta}^p: L^p(\Omega, \mathcal{F}, \mathbb{P}) \to \mathbb{R}$ given by
$$
\mathrm{EE}_{\eta}^p[Y] := \Big( \mathbb{E}\big[\max \lbrace Y - \eta, 0 \rbrace^p \big] \Big)^{1/p}.
$$
$\mathrm{EE}_{\eta}^p$ is law-invariant, convex and nondecreasing, but not translation-equivariant and positively homogeneous (cf. \cite[Example 6.22]{ShapiroDentchevaRuszczynski2014}).

\medskip\noindent
c. The \textbf{mean upper semideviation of order $p \in [1, \infty)$} is the mapping \\ $\mathrm{SD}^p_\rho: L^p(\Omega, \mathcal{F}, \mathbb{P}) \to \mathbb{R}$ defined by
$$
\mathrm{SD}_\rho^p[Y] := \mathbb{E}[Y] + \rho \cdot \mathrm{EE}_{\mathbb{E}[Y]}^p[Y] = \mathbb{E}[Y] + \rho \cdot \Big( \mathbb{E}\big[\max \lbrace \mathbb{E}[Y] - \eta, 0 \rbrace^p \big] \Big)^{1/p},
$$
where $\rho \in (0,1]$ is a parameter. $\mathrm{SD}_\rho^p$ is a law-invariant coherent risk measure (cf. \cite[Example 6.20]{ShapiroDentchevaRuszczynski2014}).

\medskip\noindent
d. The \textbf{excess probability} $\mathrm{EP}_\eta: L^0(\Omega, \mathcal{F}, \mathbb{P}) \to \mathbb{R}$ over a prescribed target level $\eta \in \mathbb{R}$ given by
$$
\mathrm{EP}_\eta[Y] = \mathbb{P}[\lbrace \omega \in \Omega \; | \; Y(\omega) > \eta \rbrace],
$$
is nondecreasing and law-invariant. However, it lacks convexity, translation-equiva-riance and positive homogeneity (cf. \cite[Example 2.29]{Claus2016}).

\medskip\noindent
e. The \textbf{Value-at-Risk} $\mathrm{VaR}_\alpha: L^0(\Omega, \mathcal{F}, \mathbb{P}) \to \mathbb{R}$ at level $\alpha \in (0,1)$ defined by
$$
\mathrm{VaR}_{\alpha}[Y] := \inf \lbrace \eta \in \mathbb{R} \; | \; \mathbb{P}[\lbrace \omega \in \Omega \; | \; Y(\omega) \leq \eta \rbrace] \geq \alpha \rbrace
$$
is law-invariant, nondecreasing, translation-equivariant and positively homogeneous, but in general not convex (cf. \cite{Pflug2000}).

\medskip\noindent
f. The \textbf{Conditional Value-at-Risk} $\mathrm{CVaR}_\alpha: L^1(\Omega, \mathcal{F}, \mathbb{P}) \to \mathbb{R}$ at level $\alpha \in (0,1)$ given by 
$$
\mathrm{CVaR}_{\alpha}[Y] := \inf \lbrace \eta + \frac{1}{1-\alpha}\mathrm{EE}_{\eta}^1[Y] \; | \; \eta \in \mathbb{R} \rbrace
$$
is a law-invariant coherent risk measure (cf. \cite[Proposition 2]{Pflug2000}). The variational representation above was established in \cite[Theorem 10]{RockafellarUryasev2002}.

\medskip\noindent
g. The \textbf{entropic risk measure} $\mathrm{Entr}_{\alpha}: L^\infty(\Omega, \mathcal{F}, \mathbb{P}) \to \mathbb{R}$ defined by
$$
\mathrm{Entr}_{\alpha}[Y] := \frac{1}{\alpha} \ln \Big( \mathbb{E} \big[\exp(\alpha Y) \big] \Big),
$$
where $\alpha > 0$ is a parameter, is a law-invariant convex (but not coherent) risk measure (cf. \cite[Example 4.13, Example 4.34]{FoellmerSchied2011}).

\medskip\noindent
h. The \textbf{worst-case risk measure} $\mathcal{R}_{\max}: L^\infty(\Omega, \mathcal{F}, \mathbb{P}) \to \mathbb{R}$ given by
$$
\mathcal{R}_{\max}[Y] := \sup_{\omega \in \Omega} Y(\omega)
$$
is law-invariant and coherent (cf. \cite[Example 4.8]{FoellmerSchied2011}). This choice of $\mathcal{R}$ in \eqref{RSBLP} leads to the \textbf{bilevel robust problem}
$$
\min_x \left\{\mathcal{R}_{\max}[F(x)] \; | \; x \in X \right\}.
$$
Note that $\mathcal{R}_{\max}$ does only depend on the so called \textbf{uncertainty set} $Z(\Omega) \subseteq \mathbb{R}$. Thus, bilevel robust problem can be formulated without knowledge of the distribution of the uncertain parameter. In robust optimization, the uncertainty set is often assumed to be finite, polyhedral or ellipsoidal (cf. \cite{BenTalElGhaouiNemirovski2009}).
\end{example}

\begin{remark}
The set of convex (coherent) risk measures on $L^p(\Omega, \mathcal{F}, \mathbb{P})$ is a convex cone for any fixed $p \in [1, \infty]$. In particular, if $\mathcal{R}: L^p(\Omega, \mathcal{F}, \mathbb{P}) \to \mathbb{R}$ is a convex (coherent) risk measure, then so is $\mathbb{E} + \rho \cdot \mathcal{R}$ for any $\rho \geq 0$. The \textbf{mean-risk bilevel stochastic programming model}
$$
\min_x \left\{\mathbb{E}[F(x)] + \rho \cdot \mathcal{R}[F(x)] \; | \; x \in X \right\}.
$$
seeks to minimize a weighted sum of the expected value of the outcome and a quantification of risk.
\end{remark}

\begin{example}{Example}
Consider the bilevel stochastic problem
$$
\min \left\{ \mathcal{R}\big[\min \Psi(x,Z) \big] \; | \; 1 \leq x \leq 6 \right\},
$$
$$
\Psi(x,z) := \mathrm{Argmin}_y \lbrace -y \; | \; y \geq 1, \; y \leq x + 2 + z_1, \; y \leq -x + 8.5 + z_2 \rbrace
$$
and assume that $Z$ is uniformly distributed over the square $[-0.5,0.5]^2$.

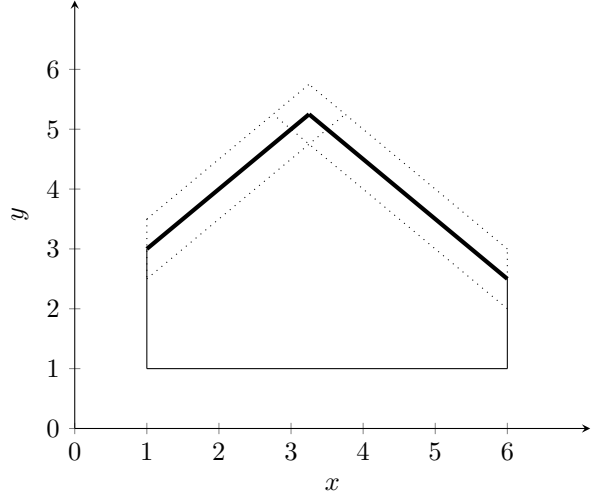
\begin{figure}
\centering
\begin{tikzpicture}[baseline]
\begin{axis}[
xtick={0,1,2,3,4,5,6},
xticklabels={$0$, $1$, $2$, $3$, $4$, $5$, $6$},
ytick={0,1,2,3,4,5,6},
yticklabels={$0$, $1$, $2$, $3$, $4$, $5$, $6$},
xmin=0, xmax=6.5,
ymin=0, ymax=6.5,
xlabel ={$x$},
ylabel ={$y$},
axis x line=bottom, 
axis y line=left,
enlargelimits=upper]
\addplot[color=black] coordinates{(1,1) (6,1)};
\addplot[color=black] coordinates{(1,1) (1,3)};
\addplot[color=black] coordinates{(6,2.5) (6,1)};
\addplot[color=black, ultra thick] coordinates{(1,3) (3.25,5.25)};
\addplot[color=black, ultra thick] coordinates{(3.25,5.25) (6,2.5)};
\addplot[color=black, dotted] coordinates{((1,2.5) (3.75,5.25)};
\addplot[color=black, dotted] coordinates{((6,2) (2.75,5.25)};
\addplot[color=black, dotted] coordinates{((6,3) (3.25,5.75)};
\addplot[color=black, dotted] coordinates{((1,3.5) (3.25,5.75)};
\addplot[color=black, dotted] coordinates{((1,3.5) (1,2.5)};
\addplot[color=black, dotted] coordinates{((6,2.5) (6,3)};
\end{axis}
\end{tikzpicture}
\caption{The bold line depicts the graph of $\Psi(\,\cdot\,, (0, 0))$, while the dotted lines correspond to the graphs of $\Psi(\,\cdot\,, (\pm 0.5, \pm0.5))$ and $\Psi(\,\cdot\,, (\mp 0.5, \pm0.5))$.}
\end{figure}

We have
$$
\Psi(x,z) = \Bigg\{ \begin{matrix} \lbrace x + 2 +z_1 \rbrace &\text{if} \;  x \leq 3.25 + 0.5z_2 - 0.5z_1 \\ \lbrace -x + 8.5 + z_2 \rbrace &\text{else} \end{matrix}
$$
for any $x \in [1,6]$ and $z \in [-0.5,0.5]^2$. A straightforward calculation shows that
\begin{align*}
\mathbb{E}\big[\min \Psi(x,Z) \big] &= \int_{-0.5}^{0.5} \int_{-0.5}^{0.5} x + 2 + z_1~dz_1~dz_2 \\
&= x+2
\end{align*}
holds for any $x \in [1,2.75]$. Similarly, we have
\begin{align*}
\mathbb{E}\big[\min \Psi(x,Z) \big] &= \int_{-0.5}^{2x-6} \int_{-0.5}^{-2x + 6.5 + z_2} x + 2 + z_1~dz_1~dz_2 \\
&+ \int_{2x-6}^{0.5} \int_{-0.5}^{0.5} x + 2 + z_1~dz_1~dz_2 \\
&+\int_{6-2x}^{0.5} \int_{-0.5}^{2x - 6.5 + z_1} -x + 8.5 + z_2~dz_2~dz_1 \\
&= -\frac{4}{3}x^3 + 11x^2 - \frac{117}{4} x + \frac{1427}{48}
\end{align*}
for $x \in [2,75,3.25]$ and
\begin{align*}
\mathbb{E}\big[\min \Psi(x,Z) \big] &= \int_{2x-7}^{0.5} \int_{-0.5}^{-2x + 6.5 + z_2} x + 2 + z_1~dz_1~dz_2 \\
&+\int_{-0.5}^{7-2x} \int_{-0.5}^{2x - 6.5 + z_1} -x + 8.5 + z_2~dz_2~dz_1 \\
&+\int_{7-2x}^{0.5} \int_{-0.5}^{0.5} -x + 8.5 + z_2~dz_2~dz_1 \\
&= \frac{4}{3} x^3 - 15 x^2 + \frac{221}{4} x - \frac{989}{16}
\end{align*}
for $x \in [3.25,3.75]$. Finally, for $x \in [3.75,6]$ we calculate
\begin{align*}
\mathbb{E}\big[\min \Psi(x,Z) \big] &= \int_{-0.5}^{0.5} \int_{-0.5}^{0.5} -x + 8.5 + z_2~dz_2~dz_1 \\
&= -x + 8.5.
\end{align*}
Thus, $\mathbb{E}\big[\min \Psi(\cdot,Z) \big]$ is piecewise polynomial, non-convex and non-differentiable. It is easy to check that $x^\ast = 6$ is a global minimizer of the risk-neutral model
$$
\min \left\{ \mathbb{E}\big[\min \Psi(x,Z) \big] \; | \; 1 \leq x \leq 6 \right\}.
$$
In this particular example, $x^\ast$ is also a global minimizer of the bilevel robust problem
$$
\min \left\{ \mathcal{R}_{max}\big[\min \Psi(x,Z) \big] \; | \; 1 \leq x \leq 6 \right\}.
$$
\end{example}

\subsection{Continuity and Differentiability} \label{SubSecContinuityDifferentiability}

Continuity properties of $\mathcal{R}$ carry over to Lipschitzian properties of $\mathcal{Q}_\mathcal{R}: P_Z \to \mathbb{R}$, $\mathcal{Q}_{\mathcal{R}}(x) := \mathcal{R}[F(x)]$.

\begin{proposition} \label{PropLipschitz}
Assume $\mathrm{dom} \; f \neq \emptyset$ and $\mu_Z \in \mathcal{M}^p_s$ for some $p \in [1,\infty]$. Then the following statements hold true for any $\mathcal{R}: L^p(\Omega, \mathcal{F}, \mathbb{P}) \to \mathbb{R}:$
\begin{enumerate}[a.]
\item $\mathcal{Q}_\mathcal{R}$ is locally Lipschitz continuous if $\mathcal{R}$ is convex and continuous.
\item $\mathcal{Q}_\mathcal{R}$ is locally Lipschitz continuous if $\mathcal{R}$ is convex and nondecreasing.
\item $\mathcal{Q}_\mathcal{R}$ is locally Lipschitz continuous if $\mathcal{R}$ is a convex risk measure.
\item $\mathcal{Q}_\mathcal{R}$ is Lipschitz continuous if $\mathcal{R}$ is Lipschitz continuous.
\item $\mathcal{Q}_\mathcal{R}$ is Lipschitz continuous if $\mathcal{R}$ is a coherent risk measure.
\end{enumerate}
\end{proposition}

\begin{proof}
a. It is well-known that any real-valued convex and continuous mapping on a normed space is locally Lipschitz continuous (cf. \cite{EkelandTemam1974}). The result is thus an immediate consequence of Lemma~\ref{LemmaF}.

\medskip\noindent
b. Any real-valued, convex and nondecreasing functional on the Banach lattice $L^p(\Omega, \mathcal{F}, \mathbb{P})$ is continuous (see e.g. \cite[Theorem 4.1]{CheriditoLi2009}).

\medskip\noindent
c. By definition, any convex risk measure is convex and nondecreasing.

\medskip\noindent
d. This is a straightforward conclusion from Lemma~\ref{LemmaF}.

\medskip\noindent
e. Any coherent risk measure on $L^p(\Omega, \mathcal{F}, \mathbb{P})$ is Lipschitz continuous by \cite[Lemma 2.1]{Inoue2003}.
\end{proof}

\begin{remark}
Any coherent risk measure $\mathcal{R}: L^\infty (\Omega, \mathcal{F}, \mathbb{P}) \to \mathbb{R}$ is Lipschitz continuous with constant 1 by \cite[Lemma 4.3]{FoellmerSchied2011}. Concrete Lipschitz constants for continuous coherent law-invariant risk measures $\mathcal{R}: L^p (\Omega, \mathcal{F}, \mathbb{P}) \to \mathbb{R}$ with $p \in [1, \infty)$ may be obtained from representation results (see e.g. \cite{BelomestnyKraetschmer2012}).
\end{remark}

Proposition~\ref{PropLipschitz} allows to formulate sufficient conditions for the existence of optimal solutions to the bilevel stochastic linear program \eqref{RSBLP}:

\begin{corollary}
Assume $\mathrm{dom} \; f \neq \emptyset$, $\mu_Z \in \mathcal{M}^p_s$ for some $p \in [1,\infty]$ and let $X \subseteq P_Z$ be nonempty and compact. Then \eqref{RSBLP} is solvable for any convex and nondecreasing mapping $\mathcal{R}: L^p(\Omega, \mathcal{F}, \mathbb{P}) \to \mathbb{R}$.
\end{corollary}

Due to the lack of convexity, Proposition~\ref{PropLipschitz} and the subsequent Corollary do not apply to the excess probability and the Value-at-Risk. However, invoking Lemma~\ref{Lemmaf}, the arguments used in the proof of \cite[Proposition 3.3]{SchultzTiedemann2003} can adapted to the setting of bilevel stochastic linear programming:

\begin{proposition}
Assume $\mathrm{dom} \; f \neq \emptyset$ and fix $\eta \in \mathbb{R}$, then $\mathcal{Q}_{\mathrm{EP}_\eta}$ is lower semicontinuous on $P_Z$ and continuous at any $x \in P_Z$ satisfying
$$
\mu_Z[\lbrace z \in \mathbb{R} \; | \; f(x,z) = \eta \rbrace] = 0.
$$
Furthermore, let $X \subseteq P_Z$ be nonempty and compact. Then
$$
\min_x \left\{ \mathrm{EP}_\eta [F(x)] \; | \; x \in X \right\}
$$
is solvable.
\end{proposition}

$\mathcal{Q}_{\mathrm{VaR}_\alpha}$ has been analyzed in \cite[Theorem 2]{Ivanov2014}:

\begin{proposition}
Assume $\mathrm{dom} \; f \neq \emptyset$ and $\alpha \in (0,1)$, then $\mathcal{Q}_{\mathrm{VaR}_\alpha}$ is continuous. Moreover, let $X \subseteq P_Z$ be nonempty and compact. Then
$$
\min_x \left\{\mathrm{VaR}_\alpha[F(x)] \; | \; x \in X \right\}
$$
is solvable.
\end{proposition}

For specific risk measures, sufficient conditions for differentiability of $\mathcal{Q}_\mathcal{R}$ have been investigated in \cite{BurtscheidtClausDempe2019}.

\begin{proposition} \label{PropDifferentiability}
Assume $\mathrm{dom} \; f \neq \emptyset$ and that $\mu_Z \in \mathcal{M}^1_s$ is absolutely continuous with respect to the Lebesgue measure. Fix any $\eta \in \mathbb{R}$, then $Q_\mathbb{E}$ and $Q_{\mathrm{EE}_\eta^1}$ are continuously differentiable at any $x_0 \in \mathrm{int} \; P_Z$. Furthermore, for any $\rho \in [0,1)$, $Q_{\mathrm{SD}_\rho^1}$ is continuously differentiable at any $x_0 \in \mathrm{int} \; P_Z$ satisfying $Q_\mathbb{E}(x_0) \neq 0$.
\end{proposition}

\begin{remark}
Theorems 3.7, 3.8 and 3.9 in \cite{BurtscheidtClausDempe2019} provide more involved sufficient conditions for continuous differentiability of $Q_\mathbb{E}$, $Q_{\mathrm{EE}_\eta^1}$ and $Q_{\mathrm{SD}_\rho^1}$ that do not require $\mu_Z$ to be absolutely continuous.
\end{remark}

\begin{remark}
Note that the assumptions of Proposition~\ref{PropDifferentiability} are not fulfilled in the example at the end of Subsect.~\ref{SubSecBilevelStochasticLinearModels}: The right-hand side of the restriction system is only \textbf{partially random} as the right-hand side of the restriction $y \geq 1$ does not depend on $Z$. If we extended system to
$$
y \leq x + 2 + Z_1', \; y \leq -x + 8.5 + Z_2', \; y \geq 1 + + Z_3',
$$
the third component of the extended random vector $Z'$ has to take the value $0$ with probability $1$. Thus, $\mathbb{P} \circ Z'^{-1}$ is not absolutely continuous with respect to the Lebesgue measure.
\end{remark}

In the presence of differentiability, necessary  optimality conditions for \eqref{RSBLP} can be formulated in terms of directional derivatives (cf. \cite[Corollary 3.10]{BurtscheidtClausDempe2019}).

\begin{proposition}
Assume $\mathrm{dom} \; f \neq \emptyset$, $\mu_Z \in \mathcal{M}^p_s$ and $X \subseteq P_Z$. Furthermore, let $x_0 \in X$ be a local minimizer of problem \eqref{RSBLP} and assume that $Q_\mathcal{R}$ is differentiable at $x_0$. Then
$$
Q'_\mathcal{R}(x_0)v \geq 0
$$
holds for any feasible direction
$$
v \in \lbrace v \in \mathbb{R}^n \; | \; \exists \epsilon_0 > 0: \; x_0 + \epsilon v \in X \; \forall \epsilon \in [0,\epsilon_0] \rbrace.
$$
\end{proposition}

\subsection{Stability} \label{SubSecStability}

While we have only considered $\mathcal{Q}_{\mathcal{R}}$ as a functions of the leader's decision $x$ so far, it also depends on the underlying probability measure $\mu_Z$. In stochastic programming, incomplete information about the true underlying distribution or the need for computational efficiency may lead to optimization models that employ an approximation of $\mu_Z$. This section analysis deals with the behaviour of optimal values and (local) optimal solution sets of \eqref{RSBLP} under perturbations of the underlying distribution.

\medskip

Taking into account that the support of the perturbed measure may differ from the original support, we shall assume $\mathrm{dom} \; f \neq \emptyset$ and
$$
P = \mathbb{R}^n \times \mathbb{R}^s
$$
to ensure that the objective function of \eqref{RSBLP} remains well defined. The corresponding assumption in two-stage stochastic programming is called \textbf{complete recourse} (cf. \cite[Sect. 2.1.3]{ShapiroDentchevaRuszczynski2014}). Sufficient conditions for $\mathrm{dom} \; f \neq \emptyset$ and $P = \mathbb{R}^n \times \mathbb{R}^s$ are given in \cite[Corollary 1]{Ivanov2014} and \cite[Corollary 2]{Ivanov2014}. The following characterization is a direct consequence of Gordan's Lemma (cf. \cite{Gordan1873}):

\begin{lemma}
$P = \mathbb{R}^n \times \mathbb{R}^s$ holds if and only if $u = 0$ is the only non-negative solution to $A^\top u = 0$.
\end{lemma}

Throughout this section, we shall consider the situation that $\mathcal{R}: L^p(\Omega, \mathcal{F}, \mathbb{P}) \to \mathbb{R}$ with $p \in [1,\infty)$ is law-invariant, convex and nondecreasing. Furthermore, for the sake of notational simplicity (cf. \cite[Remark 4.1]{BurtscheidtClausDempe2019}), we assume that the probability space $(\Omega, \mathcal{F}, \mathbb{P})$ is atomless, i.e. for any $A \in \mathcal{F}$ with $\mathbb{P}[A] > 0$ there exists some $B \in \mathcal{F}$ with $B \subsetneq A$ and $\mathbb{P}[A] > \mathbb{P}[B] > 0$.

\medskip

Then for any $x \in X$ and $\mu \in \mathcal{M}^p_s$, we have $(\delta_x \otimes \mu) \circ f^{-1} \in \mathcal{M}^p_1$, where $\delta_x \in \mathcal{P}(\mathbb{R}^n)$ denotes the Dirac measure at $x$. The atomlessness of $(\Omega, \mathcal{F}, \mathbb{P})$ ensures that there exists some $Y_{(x,\mu)} \in L^p(\Omega, \mathcal{F}, \mathbb{P})$ such that $\mathbb{P} \circ Y_{(x,\mu)}^{-1} = (\delta_x \otimes \mu) \circ f^{-1}$. Thus, we may consider the mapping $\mathcal{Q}_\mathcal{R}: X \times \mathcal{M}^p_s \to \mathbb{R}$ defined by
$$
\mathcal{Q}_\mathcal{R}(x,\mu) := \mathcal{R}[Y_{(x,\mu)}].
$$
Note that the specific choice of $Y_{(x,\mu)}$ does not matter due to the law-invariance of $\mathcal{R}$.

\medskip

Consider the parametric optimization problem
\begin{equation}
\label{ParametricProblem}
\tag{$\mathrm{P}_\mu$}
\min_x \lbrace \mathcal{Q}_\mathcal{R}(x,\mu) \; | \; x \in X \rbrace.
\end{equation}
As \eqref{ParametricProblem} may be non-convex, we shall pay special attention to sets of local optimal solutions. For any open set $V \subseteq \mathbb{R}^n$ we introduce the localized optimal value function $\varphi_V: \mathcal{M}^p_s \to \overline{\mathbb{R}}$,
$$
\varphi_V(\mu) := \min_x \lbrace \mathcal{Q}_\mathcal{R}(x,\mu) \; | \; x \in X \cap \mathrm{cl} \; V \rbrace,
$$
as well as the localized optimal solution set mapping $\phi_V: \mathcal{M}^p_s \rightrightarrows \mathbb{R}^n$,
$$
\phi_V(\mu) := \underset{x}{\mathrm{Argmin}} \lbrace \mathcal{Q}_\mathcal{R}(x,\mu) \; | \; x \in X \cap \mathrm{cl} \; V \rbrace.
$$
It is well known that additional assumptions are needed when studying stability of local solutions.

\begin{definition}
Given $\mu \in \mathcal{M}^p_s$ and an open set $V \subseteq \mathbb{R}^n$, $\phi_V(\mu)$ is called a \textbf{complete local minimizing (CLM) set} of \eqref{ParametricProblem} w.r.t. $V$ if $\emptyset \neq \phi_V(\mu) \subseteq V$.
\end{definition}

\begin{remark}
The set of global optimal solutions $\phi_{\mathbb{R}^n}(\mu)$ and any set of isolated minimizers are CLM sets. However, sets of strict local minimizers may fail to be CLM sets (cf. \cite{Robinson1987}).
\end{remark}

In the following, we shall equip $\mathcal{P}(\mathbb{R}^s)$ with the topology of weak convergence, i.e. the topology where a sequence $\{\mu_l\}_{l \in \mathbb{N}} \subset \mathcal{P}(\mathbb{R}^s)$ converges weakly to $\mu \in \mathcal{P}(\mathbb{R}^s)$, written $\mu_l \stackrel{w}{\rightarrow} \mu$, if and only if
$$
	\lim_{l \to \infty} \int_{\mathbb{R}^s} h(t)~\mu_l(dt) = \int_{\mathbb{R}^s} h(t)~\mu(dt)
$$
holds for any bounded continuous function $h:\mathbb{R}^s \to \mathbb{R}$ (cf. \cite{Billingsley1968}). The example below (cf. \cite[Example 3.2]{Claus2016}) shows that even $\varphi_{\mathbb{R}^n}$ may fail to be weakly continuous on the entire space $\mathcal{P}(\mathbb{R}^s)$.

\begin{example}{Example}
The problem
$$
\min_x \left\{x + \int_{\mathbb{R}} z~\mu(dz) \; | \; 0 \leq x \leq 1 \right\}
$$
arises from a bilevel stochastic linear problem, where $\mathcal{R} = \mathbb{E}$ and $\Psi(x,z) = \{z\} \subsetneq \mathbb{R}$ holds for any $(x,z)$. Assume that $\mu$ is the Dirac measure at $0$, then the above problem can be rewritten as 
$$
\min_x \{x \; | \; 0 \leq x \leq 1\}
$$
and its optimal value is $0$.
\medskip

However, while the sequence $\mu_l := (1 - \frac{1}{l}) \delta_{0} + \frac{1}{l} \delta_l$ converges weakly to $\delta_0$, replacing $\mu$ with $\mu_l$ yields the problem
$$
\min_x \left\{x + 1 \; | \; 0 \leq x \leq 1 \right\},
$$
whose optimal value is equal to $1$ for any $l \in \mathbb{N}$.
\end{example}

We shall follow the approach of \cite{BurtscheidtClausDempe2019}, \cite{Claus2016} and \cite{ClausKraetschmerSchultz2017} and confine the stability analysis to locally uniformly $\|\cdot\|^p$-integrating sets.

\begin{definition}
A set $\mathcal{M} \subseteq \mathcal{M}^p_s$ is said to be \textbf{locally uniformly $\|\cdot\|^p$-integrating} if for any $\epsilon > 0$ there exists some open neighborhood $\mathcal{N}$ of $\mu$ w.r.t. the topology of weak convergence such that
$$
\lim_{a \to \infty} \sup_{\nu \in \mathcal{M} \cap \mathcal{N}} \int_{\mathbb{R}^s \setminus a \mathbb{B}} \|z\|^p~\nu(dz) \leq \epsilon.
$$
\end{definition}

A detailed discussion of locally uniformly $\|\cdot\|^p$-integrating sets and their generalizations is provided in \cite{FoellmerSchied2011}, \cite{KraetschmerSchiedZaehle2012}, \cite{KraetschmerSchiedZaehle2014}, and \cite{KraetschmerSchiedZaehle2017}. The following examples demonstrate the relevance of the concept.

\begin{example}{Examples}
a. Fix $\kappa, \epsilon > 0$. Then by \cite[Corollary A.47, (c)]{FoellmerSchied2011}, the set
$$
\mathcal{M}(\kappa, \epsilon) := \left\{ \mu \in \mathcal{P}(\mathbb{R}^s) \; | \; \int_{\mathbb{R}^s} \|z\|^{p + \epsilon}~\mu(dz) \leq \kappa \right\}
$$
of Borel probability measures with uniformly bounded moments of order $p + \epsilon$ is locally uniformly $\|\cdot\|^p$-integrating.

\medskip\noindent
b. Fix any compact set $\Xi \subset \mathbb{R}^s$. By \cite[Corollary A.47, (b)]{FoellmerSchied2011}, the set
$$
\{\mu \in \mathcal{P}(\mathbb{R}^s) \; | \; \mu[\Xi] = 1\}
$$
of Borel probability measures whose support is contained in $\Xi$ is locally uniformly $\|\cdot\|^p$-integrating.
\end{example}
The following result has been established in \cite[Theorem 4.7]{BurtscheidtClausDempe2019}:

\begin{theorem} \label{TheoremStability}
Assume $\mathrm{dom} \; f \neq \emptyset$ and $P = \mathbb{R}^n \times \mathbb{R}^s$. Let $\mathcal{M} \subseteq \mathcal{M}^{p}_{s}$ be locally uniformly $\|\cdot\|^{p}$-integrating, then
\begin{enumerate}[a.]
\item $\mathcal{Q}_\mathcal{R}|_{\mathbb{R}^n \times \mathcal{M}}$ is real-valued and weakly continuous.	
\item $\varphi_{\mathbb{R}^n}|_\mathcal{M}$ is weakly upper semicontinuous.
\end{enumerate}

\noindent
In addition, assume that $\mu_0 \in \mathcal{M}$ is such that $\phi_V(\mu_0)$ is a CLM set of $P_{\mu_0}$ w.r.t. some open bounded set $V \subsetneq \mathbb{R}^n$. Then the following statements hold true:
\begin{enumerate}[c.]
\item $\varphi_V|_\mathcal{M}$ is weakly continuous at $\mu_0$.
\item[d.] $\phi_V|_\mathcal{M}$ is weakly upper semicontinuous at $\mu_0$ in the sense of Berge (cf. \cite{Berge1959}), i.e. for any open set $\mathcal{O} \subseteq \mathbb{R}^{n}$ with $\phi|_V(\mu_0) \subseteq \mathcal{O}$ there exists a weakly open neighborhood $\mathcal{N}$ of $\mu_0$ such that $\phi_V(\mu) \subseteq \mathcal{O}$ for all $\mu \in \mathcal{N} \cap \mathcal{M}$.
\item[e.] There exists some weakly open neighborhood $\mathcal{U}$ of $\mu_0$ such that $\phi_V(\mu)$ is a CLM set for \eqref{ParametricProblem} w.r.t. $V$ for any $\mu \in \mathcal{U} \cap \mathcal{M}$.
\end{enumerate}
\end{theorem}

\begin{proof}
Fix any $x_0 \in \mathbb{R}^n$. By Lemma~\ref{Lemmaf}, $f$ is Lipschitz continuous on $\mathbb{R}^n \times \mathbb{R}^s$. Thus, there exists a constant $L > 0$ such that
$$
	|f(x,z)| \leq L \|z\| + L\|x-x_0\| + |f(x_0, 0)|
$$
and the result follows from \cite[Corollary 2.4.]{ClausKraetschmerSchultz2017}.
\end{proof}

\begin{remark}
Under the assumptions of Theorem~\ref{TheoremStability}d., any accumulation point $x$ of a sequence local optimal solutions $x_l \in \phi_V(\mu_l)$ as $\mu_l \stackrel{w}{\rightarrow} \mu$, $\mu_l \in \mathcal{M}$, is a local optimal solution of \eqref{ParametricProblem}. A detailed discussion of Berge's notion of upper semicontinuity and related concepts is provided in \cite[Chap. 5]{RockafellarWets2009}.
\end{remark}

As any Borel probability measure is the weak limit of a sequence of measures having finite support, Theorem~\ref{TheoremStability} justifies an approach where the true underlying measure is approximated by a sequence of finite discrete ones. It is well known that approximation schemes based on discretization via empirical estimation (\cite{Pollard1984}, \cite{ShorackWellner1986}) or conditional expectations (\cite{BirgeWets1986}, \cite{KallRuszczynskiFrauendorfer1988}) produce weakly converging sequences of discrete probability measures under mild assumptions.

\medskip
\textbf{Attention} \textit{All results of Subsects.~\ref{SubSecPreliminaries} to \ref{SubSecStability} can be easily extended to the pessimistic approach to bilevel stochastic linear programming, where $f$ takes the form
$$
f(x,z) = c^\top x - \min_y \{- q^\top y \; | \; y \in \Psi(x,z)\}
$$
(cf. \cite[Chap. 4]{Claus2016}).}

\subsection{Stochastic Dominance Constraints}
\label{SubSecStoDom}

One possibility to model the minimization in
$$
\min \lbrace f(x,Z(\cdot)) \; | \; x \in X \rbrace 
$$
is doing it w.r.t. some risk measure that maps $f(x,Z(\cdot))$ into the reals, as introduced in Sect.~\ref{SecBilevelStochasticLinearOptimization}. In this section, we shall discuss an alternate approach, where a disutility function $g: \mathbb{R}^n \to \mathbb{R}$ is minimized over some subset of random variables of acceptable risk:
$$
	\min_x \left\{g(x) \; | \; x \in X,\; f(x, Z(\cdot)) \in \mathcal{A}\right\},
$$
where $\mathcal{A} \subseteq f(X,Z) := \{f(x, Z(\cdot)) \; | \; x \in X\}$. The following cases are of particular interest (cf. \cite[pp. 90--91]{ShapiroDentchevaRuszczynski2014}) :

\begin{example}{Examples}
a. $\mathcal{A}$ is given by \textbf{probabilistic constraints}, i.e.
$$
\mathcal{A} = \{h \in f(X,Z) \; | \; \mathbb{P}[h \leq \beta_j] \geq p_j \; \forall j= 1, \ldots, l\}
$$
for bounds $\beta_1, \ldots, \beta_l \in \mathbb{R}$ and safety levels $p_1, \ldots, p_l \in (0,1)$.

\medskip\noindent
b. $\mathcal{A}$ is given by \textbf{first-order stochastic dominance constraints}, i.e.
$$
\mathcal{A} = \{h \in f(X,Z) \; | \; \mathbb{P}[h \leq \beta] \geq \mathbb{P}[b \leq \beta] \; \forall \beta \in \mathbb{R}\},
$$
where $b \in L^{0}(\Omega, \mathcal{F}, \mathbb{P})$ is a given benchmark variable. If $b$ is discrete with a finite number of realizations, it is sufficient to impose the relation $\mathbb{P}[h \leq \beta] \geq  \mathbb{P}[b \leq \beta]$ for any $\beta$ in a finite subset of $\mathbb{R}$. In this case, $\mathcal{A}$ admits a description by a finite system of probabilistic constraints.

\medskip\noindent
c. $\mathcal{A}$ is given by \textbf{second-order stochastic dominance constraints}, i.e.
$$
\mathcal{A} = \{h \in f(X,Z) \; | \; \mathbb{E}[\max\{h - \eta, 0\}] \leq \mathbb{E}[\max\{b - \eta, 0\}] \; \forall \eta \in \mathbb{R}\},
$$
where $b \in L^{1}(\Omega, \mathcal{F}, \mathbb{P})$ is a given benchmark variable.
\end{example}

A discussion of general models involving probabilistic or stochastic dominance constraints can be found in \cite[Chap. 8]{ShapiroDentchevaRuszczynski2014} and \cite[Chap. 8.3]{Prekopa1995}.

\medskip

Let $\nu := \mathbb{P} \circ b^{-1} \in \mathcal{P}(\mathbb{R})$ denote the distribution of the benchmark variable $b$. Then the feasible set under first-order stochastic dominance constraints admits the representation
$$
\left\{ x \in X \; | \; \mu_Z\big[ \lbrace z \in \mathbb{R}^s \; | \; f(x,z) \leq \beta \rbrace \big] \geq \nu\big[ \lbrace b \in \mathbb{R} \; | \; b \leq \beta \rbrace \big] \; \forall \beta \in \mathbb{R} \right\}.
$$
Similarly, for second-order stochastic dominance constraints, $\mu \in \mathcal{M}^1_s$ and $\nu \in \mathcal{M}^1_1$, the feasible set takes the form
$$
\left\{ x \in X \, | \int_{\mathbb{R}^s} \max \lbrace f(x,z) - \eta, 0 \rbrace~\mu_Z(dz) \leq \int_{\mathbb{R}} \max \lbrace b - \eta, 0 \rbrace~\nu(db) \; \forall \eta \in \mathbb{N} \right\}.
$$
In both cases, the feasible does only depend on the distribution of the underlying random vector. As in Subsect.~\ref{SubSecStability}, we consider situations where $\mu_Z$ is replaced with an approximation and study the behaviour of the mappings $\mathcal{C}_1: \mathbb{P}(\mathbb{R}^s) \rightrightarrows \mathbb{R}^n$ defined by
$$
\mathcal{C}_1(\mu) = \left\{ x \in X \; | \; \mu \big[ \lbrace z \in \mathbb{R}^s \; | \; f(x,z) \leq \beta \rbrace \big] \geq \nu\big[ \lbrace b \in \mathbb{R} \; | \; b \leq \beta \rbrace \big] \; \forall \beta \in \mathbb{R} \right\}.
$$
and $\mathcal{C}_2: \mathcal{M}^1_s \rightrightarrows \mathbb{R}^n$ given by
\small{$$
C_2(\mu) := \left\{ x \in X \, | \int_{\mathbb{R}^s} \max \lbrace f(x,z) - \eta, 0 \rbrace~\mu(dz) \leq \int_{\mathbb{R}} \max \lbrace b - \eta, 0 \rbrace~\nu(db) \, \forall \eta \in \mathbb{N} \right\}.
$$}\normalsize

Invoking Lemma~\ref{Lemmaf}, the following result can be obtained by adapting the proofs of \cite[Proposition 2.1]{GollmerNeiseSchultz2008} and \cite[Proposition 2.2]{GollmerGotzesSchultz2011} :

\begin{proposition} \label{PropStoDom}
Assume $\mathrm{dom} \; f \neq \emptyset$ and $P = \mathbb{R}^n \times \mathbb{R}^s$. Then the following statements hold true:
\begin{enumerate}[a.]
\item The multifunction $C_1$ is closed w.r.t. the topology of weak convergence, i.e. for any sequences $\lbrace \mu_{l} \rbrace_{l} \subset \mathcal{P}(\mathbb{R}^{s})$ and $\lbrace x_{l} \rbrace_{l} \subset \mathbb{R}^{n}$ with $\mu_{l} \stackrel{w}{\rightarrow} \mu \in \mathcal{P}(\mathbb{R}^{s})$, $x_{l} \rightarrow x \in \mathbb{R}^{n}$ for $l \rightarrow \infty$ and $x_{l} \in C_1(\mu_{l})$ for all $l \in \mathbb{N}$ it holds true that $x \in C_1(\mu)$.
\item Additionally assume that $\nu \in \mathcal{M}^1_1$, then the multifunction $C_2$ is closed w.r.t. the topology of weak convergence.
\end{enumerate}
\end{proposition}

By considering the constant sequence $\mu_l = \mu$ for all $l \in \mathbb{N}$ we obtain the closedness of the sets $C_1(\mu)$ and $C_2(\mu)$ under the conditions of Proposition~\ref{PropStoDom}. The closedness of the multifunctions $C_1$ and $C_2$ is also the key to proving the following stability result (cf. \cite[Propostition 2.5]{GollmerNeiseSchultz2008}):

\begin{theorem}
Assume $\mathrm{dom} \; f \neq \emptyset$, $P = \mathbb{R}^n \times \mathbb{R}^s$ and that $X$ is nonempty and compact. Moreover, let $g$ be lower semicontinuous. Then the following statements hold true:
\begin{enumerate}[a.]
\item The optimal value function $\varphi_1: \mathcal{P}(\mathbb{R}^{s}) \to \mathbb{R} \cup \lbrace \infty \rbrace$ given by
$$
\varphi_1(\mu) := \inf \lbrace g(x) \; | \; x \in C_1(\mu)\rbrace
$$
is weakly lower semicontinuous on $\mathrm{dom} \; C_1$.
\item Additionally assume $\nu \in \mathcal{M}^1_1$, then the function $\varphi_2: \mathcal{M}^1_s \to \mathbb{R} \cup \lbrace \infty \rbrace$ given by
$$
\varphi_2(\mu) := \inf \lbrace g(x) \; | \; x \in C_2(\mu)\rbrace
$$
is weakly lower semicontinuous on $\mathrm{dom} \; C_2$.
\end{enumerate}
\end{theorem}

\subsection{Finite Discrete Distributions}
\label{SubSecFiniteDiscreteDistributions}

Throughout this section, we shall assume that the underlying random vector $Z$ is discrete with a finite number of realizations $Z_1, \ldots, Z_K \in \mathbb{R}^s$ and respective probabilities $\pi_1, \ldots, \pi_K \in (0,1]$. Let $I$ denote the index set $\lbrace 1, \ldots, K \rbrace$, then $P_Z$ takes the form
$$
P_Z = \lbrace x \in \mathbb{R}^n \; | \; \forall k \in I \; \exists y \in \mathbb{R}^m: \; Ay \leq Tx + Z_k \rbrace.
$$
Suppose that $x_0 \in X$ is such that $\lbrace y \in \mathbb{R}^m \; | \; Ay \leq Tx_0 + Z_k \rbrace = \emptyset$ holds for some $k \in I$. Then the probability of $f(x_0,Z(\omega)) = \infty$ is a least $\pi_k > 0$, i.e. $x_0$ should be considered as infeasible for problem \eqref{RSBLP}. Consequently, $X \subseteq P_Z$ can be understood as an \textbf{induced constraint}. Note that $X \cap P_Z$ is a polyhedron if $X$ is a polyhedron.

\medskip

In this setting, the bilevel stochastic linear problem can be reduced to a standard bilevel program, which allows to adapt optimality conditions and algorithms designed for the deterministic case (cf. \cite{Dempe2002}). 

\begin{proposition}\label{PropositionSD_discretemodels}
Assume $\mathrm{dom} \; f \neq \emptyset$, $\mathcal{R} \in \{\mathbb{E}$, $\mathrm{EE}_\eta^1$, $\mathrm{SD}^1_\rho$, $\mathrm{EP}_\eta$, $\mathrm{VaR}_\alpha$, $\mathrm{CVaR}_\alpha$, $\mathcal{R}_{\max}\}$ and let $X \subseteq P_Z$ be a polyhedron. If $\mathcal{R} \in \lbrace \mathrm{EP}_{\eta}, \mathrm{VaR}_{\alpha} \rbrace$, additionally assume that $X$ bounded. Then for any parameter $\beta$, there exists a constant $M > 0$ such that the bilevel stochastic linear problem
$$
\min_x \left\{ \mathcal{R}[F(x)] \; | \; x \in X \right\}
$$
is equivalent to the standard bilevel program
$$
\min_x \Big\{\aunderbrace[l0D0r@{\hspace{2.6em}}]{\ \;\inf_{\eta \in \mathbb{R}}\ \min_w \{a}_{\text{if}\ \mathcal{R} = \mathrm{CVaR}_{\alpha}}(x, w) \; | \; w \in \Psi_{\mathcal{R}}(x)\}| \; x \in X\Big\},\ \text{or}
$$
$$
\min_x \Big\{c^\top x + \inf_{\eta \in \mathbb{R}}\big\{\eta \; | \; \min_w \{a(x, w) \; | \; w \in \Psi_{\mathcal{R}}(x)\} \geq \alpha\big\}| \; x \in X\Big\}\quad \text{if}\ \mathcal{R}= \mathrm{VaR}_{\alpha},
$$
where
$$
\Psi_{\mathcal{R}}(x) := \underset{w}{\mathrm{Argmin}} \; \left\{\sum_{k \in I}d^\top y_k \; | \; Ay_k \leq Tx + Z_k,\;b_k(x, w) \geq 0\ \forall k \in I\right\}.
$$
The specific formulations can be found in Table~\ref{TableRM_discretemodels}.

\renewcommand{\arraystretch}{1.2}
\begin{table}
\caption{Equivalent bilevel linear programs}
\label{TableRM_discretemodels}
\begin{tabular}{lllll}
\hline\noalign{\smallskip}
$\mathcal{R}$&$\beta$&$w$&$a(x, w)$&$b(x, w)$\\
\noalign{\smallskip}\hline\hline\noalign{\smallskip}
$\!\!\!\mathbb{E}\!\!\!$&&$\!\!\!(y_1, \ldots, y_K) \in \mathbb{R}^{Km}\!\!\!$&$\!\!\!\!\!\!c^\top x + \sum_{k \in I}\pi_k q^\top y_k\!\!\!$&\\ \hline
$\!\!\!\mathrm{EE}^1_{\eta}\!\!\!$&$\!\!\!\eta \in \mathbb{R}\!\!\!$&$\!\!\!\!\!\!\begin{array}{l} (y_1, \ldots, y_K) \in \mathbb{R}^{Km}\\(v_1, \ldots, v_K) \in \mathbb{R}^{K}\end{array}\!\!\!$&$\!\!\!\!\!\!\sum_{k \in I}\pi_k v_k\!\!\!$&$\!\!\!\!\!\!\left(\!\!\!\begin{array}{l} v_k\\v_k - c^\top x - q^\top y_k + \eta\end{array}\!\!\!\right)\!\!\!$\\ \hline
$\!\!\!\mathrm{SD}^1_\rho\!\!\!$&$\!\!\!\rho \in (0, 1]\!\!\!\!\!\!$&$\!\!\!\!\!\!\begin{array}{l} (y_1, \ldots, y_K) \in \mathbb{R}^{Km}\\(v_1, \ldots, v_K) \in \mathbb{R}^{K}\end{array}\!\!\!$&$\!\!\!\!\!\!\!\!\!\begin{array}{l}(1 - \rho)\sum_{k \in I}\pi_k q^\top y_k\\+ \rho\sum_{k \in I}\pi_k v_k + c^\top x\end{array}\!\!\!$&$\!\!\!\!\!\!\left(\!\!\!\begin{array}{l}v_k - q^\top y_k\\v_k - \sum_{j \in I}\pi_j q^\top y_j\end{array}\!\!\!\right)$\\ \hline
$\!\!\!\mathrm{EP}_\eta\!\!\!$&$\!\!\!\eta \in \mathbb{R}\!\!\!$&$\!\!\!\!\!\!\begin{array}{l} (y_1, \ldots, y_K) \in \mathbb{R}^{Km}\\(\theta_1, \ldots, \theta_K) \in \{0, 1\}^{K}\end{array}\!\!\!$&$\!\!\!\!\!\!\sum_{k \in I}\pi_k \theta_k\!\!\!$&$\!\!\!\!\!\!M\theta_k - c^\top x - q^\top y_k + \eta\!\!\!$\\ \hline
$\!\!\!\mathrm{VaR}_{\alpha}\!\!\!\!\!\!$&$\!\!\!\alpha \in (0, 1)\!\!\!$&$\!\!\!\!\!\!\begin{array}{l} (y_1, \ldots, y_K) \in \mathbb{R}^{Km}\\(\theta_1, \ldots, \theta_K) \in \{0, 1\}^{K}\end{array}\!\!\!$&$\!\!\!\!\!\!\sum_{k \in I}\pi_k \theta_k\!\!\!$&$\!\!\!\!\!\!\!\!\!\begin{array}{l}M(1 - \theta_k) - c^\top x\\ - q^\top y_k + \eta\end{array}\!\!\!$\\ \hline
$\!\!\!\mathrm{CVaR}_{\alpha}\!\!\!$&$\!\!\!\alpha \in (0,1)\!\!\!$&$\!\!\!\!\!\!\begin{array}{l} (y_1, \ldots, y_K) \in \mathbb{R}^{Km}\\(v_1, \ldots, v_K) \in \mathbb{R}^{K}\end{array}\!\!\!\!\!\!$&$\!\!\!\!\!\!\eta + \frac{1}{1 - \alpha}\sum_{k \in I}\pi_k v_k\!\!\!$&\!\!\!\!\!\!cf. $\mathrm{EE}^1_{\eta}\!\!\!$ \\ \hline
$\!\!\!\mathcal{R}_{\max}\!\!\!\!\!\!$&&$\!\!\!(y_1, \ldots, y_K) \in \mathbb{R}^{Km}\!\!\!$&$\!\!\!\!\!\!\max_{k \in I} c^\top x + q^\top y_k\!\!\!$& \\
\noalign{\smallskip}\hline\noalign{\smallskip}
\end{tabular}
\end{table}
\end{proposition}

\begin{proof}
For $\mathcal{R} \in \lbrace \mathbb{E}, \mathrm{EE}_\eta^1, \mathrm{SD}^1_\rho, \mathrm{CVaR}_\alpha \rbrace$, we refer to \cite[Section 5]{BurtscheidtClausDempe2019}.

\medskip\noindent
For the excess probability, the first of the considered quantile-based risk measures, we have $\mathrm{EP}_{\eta}[F(x)] = \mathbb{P}\left[c^\top x + \inf_y \{q^\top y \; | \; y \in \Psi(x, Z(\cdot))\} > \eta\right]$. Fix $M \in \mathbb{R}$ such that
$$
	M > \sup\left\{c^\top x + \inf_{y_k} \{q^\top y_k \; | \; y_k \in \Psi(x, Z_k)\} \; | \; x \in X,\; k \in I\right\} - \eta
$$
and, for $y_k \in \Psi(x, Z_k)$, let
$$
	\theta_k :=
	\begin{cases}
		0 & \text{if}\ c^\top x + q^\top y_k - \eta \leq 0,\\
		1 & \text{otherwise.}
	\end{cases}
$$
Then the excess probability is equal to
\begin{align*}
	&\sum_{k \in I}\pi_k \inf_{y_k, \theta_k} \left\{\theta_k \; | \; M\theta_k \geq c^\top x + q^\top y_k - \eta,\; y_k \in \Psi(x, Z_k),\; \theta_k \in \{0, 1\}\right\}\\
	&\left.
	\begin{aligned}
		= \inf_{\genfrac{}{}{0pt}{}{y_1, \ldots, y_K,}{\theta_1, \ldots, \theta_K}} \left\{\sum_{k \in I}\pi_k \theta_k \; | \; \right.&M\theta_k \geq c^\top x + q^\top y_k - \eta,\; y_k \in \Psi(x, Z_k),\; \theta_k \in \{0, 1\}\\
		&\forall k \in I
	\end{aligned}
	\right\}.
\end{align*}

\medskip\noindent
Similar to the proof for $\mathcal{R} = \mathrm{EP}_\eta$, $\mathbb{P}[f(x,Z(\cdot)) \leq \eta]$ equals
\small{$$
\begin{aligned}
	&\sum_{k \in I}\pi_k \inf_{y_k, \theta_k} \left\{\theta_k \; | \; M(1 - \theta_k) \geq c^\top x + q^\top y_k - \eta,\; y_k \in \Psi(x, Z_k),\; \theta_k \in \{0, 1\} \right\}\\
	&\left.
	\begin{aligned}
		= \left.\inf_{\genfrac{}{}{0pt}{}{y_1, \ldots, y_K,}{\theta_1, \ldots, \theta_K}} \right\{\sum_{k \in I}\pi_k \theta_k \; | \; &M(1 - \theta_i) \geq c^\top x + q^\top y_k - \eta,\; y_k \in \Psi(x, Z_k),\; \theta_k \in \{0, 1\} \\
		&\forall k \in I
	\end{aligned}
	\right\},
\end{aligned}
$$}\normalsize
where
$$
    \theta_k :=
	\begin{cases}
		1 & \text{if}\ c^\top x + q^\top y_k - \eta \leq 0,\\
		0 & \text{otherwise.}
	\end{cases}
$$
Thereby we get equality of $\mathrm{VaR}_{\eta}[f(x, Z(\cdot))]$ and
$$
	\inf \left\{\eta \in \mathbb{R} \; | \; \inf_{\genfrac{}{}{0pt}{}{y_1, \ldots, y_K,}{\theta_1, \ldots, \theta_K}} \left\{\sum_{k \in I}\pi_k \theta_k \; | \; (y_1, \ldots, y_K, \theta_1, \ldots, \theta_K) \in \Psi_{\mathrm{VaR}_{\eta}}(x) \right\} \geq \alpha\right\}.
$$

\medskip\noindent
The worst-case risk measure is equal to
$$\sup_{k \in I} \left\{c^\top x + \min_{y_k} \{q^\top y \; | \; y \in \Psi(x, Z_k)\}\right\}
$$
and the result follows from $\Psi_{\mathcal{R}_{\max}} = \Psi(x, Z_1) \times \ldots \times \Psi(x, Z_K)$.
\end{proof}

\begin{remark} \label{RemBlockStructure}
a. The equivalent standard bilevel problem is linear if $\mathcal{R} \in \lbrace \mathbb{E}, \mathrm{EE}_\eta^1, \mathrm{SD}_{\rho}^1 \rbrace$.

\medskip\noindent
b. Analogous to \cite[Remarks 5.2, 5.4]{BurtscheidtClausDempe2019}, the inner minimization problems of the standard bilevel linear programs for $\mathcal{R} \in \lbrace \mathbb{E}, \mathrm{EE}_\eta^1, \mathrm{EP}_\eta, \mathcal{R}_{\max} \rbrace$ can be decomposed into $K$ scenario problems that only differ w.r.t. the right-hand side of the constraint system. For the other models, a similar decomposition is possible after Lagrangean relaxation of the coupling constraints involving different scenarios. 

\medskip\noindent
c. For $\mathcal{R} = \mathrm{CVaR}_\alpha$, every evaluation the objective function in the standard bilevel linear program corresponds to solving a bilevel linear problem with scalar upper level variable $\eta$.

\medskip\noindent
d. Alternate models for $\mathcal{R} = \mathrm{VaR}_\alpha$ are given in \cite{Ivanov2014} and \cite{DempeIvanovNaumov2017}, where the considered bilevel stochastic linear problem is reduced to a mixed-integer nonlinear program and a mathematical programming problem with equilibrium constraints, respectively. 
A mean-risk model with $\mathcal{R} = \mathrm{CVaR}_\alpha$ is used in \cite[Sect. III]{CarrionArroyoConejo2009}.
\end{remark}

Similar reformulations can be obtained for the models discussed in Subsect.~\ref{SubSecStoDom} if we assume that the disutility function is linear.

\begin{proposition}
Assume $\mathrm{dom} \; f \neq \emptyset$, and let $X \subseteq P_Z$ be a bounded polyhedron. Then for any parameter $\gamma$, the problem
$$
	\min_x \left\{g^\top x \; | \; F(x) \in \mathcal{A},\; x \in X \right\}
$$
is equivalent to
$$
	\min_x \Big\{g^\top x \; | \; \inf_{w_j} \left\{a(w_j) \; |\; w_j \in \Psi_{\mathcal{R}}(x)\right\} \geq \delta_j \ \forall j = 1, \ldots, l,\; x \in X\Big\}.
$$
The specific formulations based on the examples in Subsect.~\ref{SubSecStoDom} are listed in Table~\ref{TableSD_discretemodels}, where $\bar{a}_j := 1 - \mathbb{P}[b \leq a_j]$ and $\tilde{a}_j := \int_{\mathbb{R}^s}\max\{b(z) - a_j, 0\}\;\mu_Z(dz)$.

\renewcommand{\arraystretch}{1.2}
\begin{table}
\caption{Equivalent programs}
\label{TableSD_discretemodels}
\begin{tabular}{llllll}
\hline\noalign{\smallskip}
$\mathcal{A}$&$\gamma$&$w_j$&$a(w_j)$&$\mathcal{R}$&$\delta_j$\\
\noalign{\smallskip}\hline\hline\noalign{\smallskip}
a.\!\!\!&$\!\!\!\!\!\!\begin{array}{l}\beta_j \in \mathbb{R},\ p_j \in (0, 1)\\\text{with}\ j = 1, \ldots, l\end{array}\!\!\!$&$\!\!\!\!\!\!\begin{array}{l} (y_{1j}, \ldots, y_{Kj}) \in \mathbb{R}^{Km} \\ (\theta_{1j}, \ldots, \theta_{Kj}) \in \lbrace 0,1 \rbrace^{K}\end{array}\!\!\!$&$\!\!\!\sum_{k \in I}\pi_k \theta_{kj}\!\!\!$&$\mathrm{VaR}_{\beta_j}\!\!\!$&$p_j\!\!\!$\\ \hline
b.\!\!\!&$\!\!\!\!\!\!\begin{array}{l}\text{any finite discrete}\\\text{benchmark variable}\end{array}\!\!\!$&$\!\!\!\!\!\!\begin{array}{l} (y_{1j}, \ldots, y_{Kj}) \in \mathbb{R}^{Km} \\ (\theta_{1j}, \ldots, \theta_{Kj}) \in \lbrace 0,1 \rbrace^{K}\end{array}\!\!\!$&$\!\!\!\sum_{k \in I}\pi_k \theta_{kj}\!\!\!$&$\mathrm{EP}_{a_j}\!\!\!$&$\bar{a}_j\!\!\!$\\\cline{1-1}\cline{3-6}
c.\!\!\!&$\!\!\!\!\!\!\begin{array}{l}b\ \text{with realizations}\\a_1, \ldots, a_l\end{array}\!\!\!$&$\!\!\!\!\!\!\begin{array}{l} (y_{1j}, \ldots, y_{Kj}) \in \mathbb{R}^{Km} \\ (v_{1j}, \ldots, v_{Kj}) \in \mathbb{R}^{K}\end{array}\!\!\!$&$\!\!\!\sum_{k \in I}\pi_k v_{kj}\!\!\!$&$\mathrm{EE}_{a_j}^1\!\!\!$&$\tilde{a}_j$\\
\noalign{\smallskip}\hline\noalign{\smallskip}
\end{tabular}
\end{table}
\end{proposition}

\subsection{Solution Approaches}

To solve bilevel problems, it is very common to use a single level reformulation. Often the lower level minimality condition is replaced by the Karush-Kuhn-Tucker or the Fritz John conditions and the bilevel problem is reduced to a mathematical programming problem with equilibrium constraints (cf. \cite{CarrionArroyoConejo2009}, \cite[Chap. 3.5.1]{Dempe2002}, \cite{Ivanov2014}).

\medskip

For $\mathcal{R} \in \lbrace \mathbb{E}, \mathrm{EE}_\eta^1, \mathrm{SD}^1_\rho, \mathrm{EP}_\eta, \mathcal{R}_{\max} \rbrace$, the equivalent standard bilevel programs in Proposition~\ref{PropositionSD_discretemodels} can be all be restated as
\begin{equation}
\label{GenForm}
\min_u \lbrace g^\top u + \min_w \lbrace h^\top w \; | \;  w \in \Psi(u) \rbrace \; | \; u \in U \rbrace,
\end{equation}
where $\Psi: \mathbb{R}^k \rightrightarrows \mathbb{R}^l$ is given by $\Psi(u) = \mathrm{Argmin}_w \lbrace t^\top w \; | \; Ww \leq Bu + b \rbrace$ for vectors $g \in \mathbb{R}^k$, $h, t \in \mathbb{R}^l$ and $b \in \mathbb{R}^r$, matrices $W \in \mathbb{R}^{r \times l}$ and $B \in \mathbb{R}^{r \times k}$, and $U \subseteq \mathbb{R}^k$ is a nonempty polyhedron. The usage of the KKT conditions of the lower level problem leads to the single-level problem
\begin{equation}
\label{KKTReform}
\min_{u, w, v} \left\{ g^\top u + h^\top w \; \Bigg| \; \begin{aligned}  &Ww \leq Bu + b, \; W^\top v = t, \; v \leq 0, \\ &v^\top (Ww - Bu - b) = 0, \; u \in U \end{aligned} \right\}.
\end{equation}
More details as well as statements on the coincidence of optimal values and the existence of local and global minimizers are given in \cite[Sect. 6]{BurtscheidtClausDempe2019}. If the condition $v^\top (Ww - Bu - b) = 0$ is relaxed by $v^\top (Ww - Bu - b) \leq \varepsilon$ (the resulting problem is denoted by $\mathrm{P(}\varepsilon\mathrm{)}$), the violation of regularity conditions like (MFCQ) and (LICQ) at every feasible point of \eqref{KKTReform} can be bypassed. A discussion of other difficulties associated with \eqref{KKTReform} is provided in \cite[Chap. 3.1.2]{Henkel2014}.

\medskip

In \cite[Sect. 6]{BurtscheidtClausDempe2019} it is also shown that $(\overline{u}, \overline{w})$ is a local minimizer of the optimistic formulation, if $(\overline{u}, \overline{w}, \overline{v})$ is an accumulation point of a sequence $\lbrace (u_n, w_n, v_n) \rbrace_{n \in \mathbb{N}}$ of local minimizers of problem $\mathrm{P(}\varepsilon_n\mathrm{)}$ for $\varepsilon_n\downarrow 0$.

\medskip

In the risk-neutral setting, problem \eqref{GenForm} exhibits a block-structure (cf. Remark~\ref{RemBlockStructure} b.). Adapting the solution method for general linear complementarity problems proposed in \cite{HuMitchellPangBennettKunapuli2008}, this special structure has been used in \cite[Chap. 6]{Henkel2014} to construct an efficient algorithm for the global resolution of bilevel stochastic linear problems based on dual decomposition.

\begin{remark}
Utilizing the lower level value function, problem \eqref{GenForm} can be reformulated as a single level quasiconcave optimization problem (cf. \cite[Chap. 3.6.5]{Dempe2002}). Solution methods based on a branch-and-bound scheme have been proposed in \cite{Tuy1998} and \cite{TuyMigdalasVaerbrand1994}. However, without modifications, these algorithms fail to exploit the block structure arising in risk-neutral bilevel stochastic linear optimization models (cf. \cite[Chap. 4.2]{Henkel2014}).
\end{remark}

\section{Two-stage Stochastic Bilevel Programs}

In two-stage stochastic bilevel programming, both leader and follower have to make their respective first-stage decisions without knowledge of the realization of a stochastic parameter. Afterwards, the second-stage decisions are made under complete information. This leads to the following chronology of decision and observation: 
$$
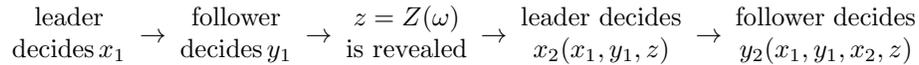
\begin{matrix} \text{leader} \\ \text{decides}\, x_1 \end{matrix}
\hspace{2.5pt} \rightarrow \hspace{2.5pt} 
\begin{matrix} \text{follower} \\ \text{decides}\, y_1 \end{matrix}
\hspace{2.5pt} \rightarrow \hspace{2.5pt}
\begin{matrix} z = Z(\omega) \\ \text{is revealed} \end{matrix}
\hspace{2.5pt} \rightarrow \hspace{2.5pt}
\begin{matrix} \text{leader decides} \\ x_2(x_1,y_1,z) \end{matrix}
\hspace{2.5pt} \rightarrow \hspace{2.5pt}
\begin{matrix} \text{follower decides} \\ y_2(x_1,y_1,x_2,z) \end{matrix}
$$

\begin{remark}
The bilevel stochastic linear problems considered in Sect.~\ref{SecBilevelStochasticLinearOptimization} can be understood as special two-stage bilevel programs, where the follower's first-stage and the leader's second stage decision do not influence the outcome.
\end{remark}

In \cite{AlizadehMarcotteSavard2013}, a two-stage stochastic extension of the bilevel network pricing model introduced in \cite{LabbeMarcotteSavard1998} is studied. Consider a multicommodity transportation network $(N, \Lambda, K)$, where $(N,\Theta)$ is a directed graph and each commodity $k \in K$ is to be transported from an origin $O(k) \in N$ to a destination $D(k) \in N$ in order to satisfy a demand $n^k \in (0,\infty)$. The set of arcs $\Theta$ is partitioned into the subsets $\theta$ and $\overline{\theta}$ of tariff a tariff-free arcs, respectively, and the leaders is maximizing the revenue raised from tariffs, knowing that user flows are assigned to cheapest paths. In \cite{LabbeMarcotteSavard1998}, this situation is modeled as a bilevel program
$$
"\max_{x}" \left\{ \sum_{k \in K} x^\top y^k \; | \; (y,\overline{y}) \in \Psi(x) \right\}
$$
with lower level is given by
$$
\Psi(x, c, d, b) := \underset{y,\overline{y}}{\mathrm{Argmin}} \left\{ \sum_{k \in K} \left[ (c + x)^\top y^k + \overline{c}^\top \overline{y}^k \right] \; | \; \begin{matrix} y, \overline{y} \geq 0, \\ Ay^k + \overline{A} \overline{y}^k = b^k \; \forall k \in K \end{matrix} \right\},
$$
where $x$ is the vector of tariffs controlled by the leader, $y^k$ and $\overline{y}^k$ are the flows of commodity $k$ on the tariff and tariff-free arcs, respectively. Moreover, $c$ and $\overline{c}$ are the fixed costs on $\theta$ and $\overline{\theta}$, respectively, $(A,\overline{A})$ denotes the node-arc incidence matrix and the vectors $b^k$ defined by
$$
b^k_i := \Bigg\{ \begin{matrix} n^k, &\text{if} \; i = O(k) \\ -n^k, &\text{if} \; i = D(k) \\ 0, &\text{else} \end{matrix}
$$
are used to express nodal balance. \cite{AlizadehMarcotteSavard2013} extends the above model to a two-stage setting including market uncertainties: After deciding on first-stage tariffs, the situation repeats itself on the same network but with different cost and demand parameters. At the first-stage, only the distribution of the second-stage parameter $Z(\omega) = (c_2, d_2, b_2)(\omega)$ is known and the stages are linked by the restriction that the second-stage tariffs should not differ too widely from those set at the first stage. The linking constraint is motivated by policy regulations and competitivity issues. In a risk-neutral setting, this results in the problem
\begin{equation}
\label{TwoStageNetworkPricing}
"\max_{x_1}" \left\{ \sum_{k \in K} x_1^\top y^k_1 + \mathbb{E}\big[\Phi(x_1, Z(\cdot))\big] \; | \; (y_1,\overline{y}_1) \in \Psi(x_1,c_1,d_1,b_1) \right\},
\end{equation}
where the recourse is given by
$$
\Phi(x_1, Z(\omega)) := "\max_{x_2}" \left\{ \sum_{k \in K} x_2^\top y^k_2 \; | \; (x_1,x_2) \in \Gamma(\delta), \; (y_2,\overline{y}_2) \in \Psi(x_2,Z(\omega)) \right\}
$$
and the set $\Gamma(\delta)$ is defined as either
$$
\Gamma(\delta) := \Gamma_A(\delta) := \lbrace (x_1,x_2) \; | \; |x_{1, \theta} - x_{2, \theta}| \leq \delta_\theta \; \forall \theta \in \Theta \rbrace
$$
if tariff changes are limited in absolute values or
$$
\Gamma(\delta) := \Gamma_R(\delta) := \lbrace (x_1,x_2) \; | \; |x_{1, \theta} - x_{2, \theta}| \leq \delta_\theta |x_{1, \theta}| \; \forall \theta \in \Theta \rbrace
$$
if proportional limits are considered. Assuming that the underlying random vector $Z$ is discrete with a finite number of realizations, a reformulation of \eqref{TwoStageNetworkPricing} as a single-stage bilevel program is established in \cite{AlizadehMarcotteSavard2013}. Moreover, sensitivity analysis of the optimal value function of \eqref{TwoStageNetworkPricing} w.r.t. the parameter $\delta \in [0,\infty)^{|\Theta|}$ (cf. \cite[Proposition 4.1, Proposition 4.2]{AlizadehMarcotteSavard2013}) as well as numerical studies are conducted (cf. \cite[Sect. 5, Sect. 6]{AlizadehMarcotteSavard2013}).

\section{Challenges}

We shall highlight some aspects of bilvel stochastic programming that are highly deserving of future research: 

\medskip
\textbf{Going (further) beyond the risk-neutral case for nonlinear models:} The first paper on bilevel stochastic programming has already outlined the basic principles as well as existence and sensitivity results for risk neutral models (cf. \cite{PatrikssonWynter1999}). Nevertheless, so far, most of the research on bilevel stochastic nonlinear programming is still concerned with the risk-neutral case. Notable exceptions are \cite{CarrionArroyoConejo2009} and \cite{KovacevicPflug2013}, where models involving the conditional value at risk are considered. In the first paper the problem of maximizing the medium-term revenue of an electricity retailer under uncertain pool prices, demand, and competitor prices is modeled as a bilevel stochastic quadratic problem, while the latter explores links between electricity swing option pricing and stochastic bilevel optimization. However, there exists no systematic analysis of bilevel stochastic nonlinear problems in the broader framework of coherent risk measures or higher stochastic dominance constraints. Future research may also consider distributionally robust models (cf. \cite{ZhangXuZhang2016}).

\medskip

\textbf{Exploiting (quasi) block structures arising in risk-averse models:} Under finite discrete distributions many bilevel stochastic problems can be reformulated as standard bilevel programs. While this reformulation entails a blow-up of the dimension which is usually linear in the number of scenarios, the resulting problems often exhibit (quasi) block structures (cf. Remark~\ref{RemBlockStructure}b., \cite{PatrikssonWynter1999}). For risk-neutral bilevel stochastic linear problems, \cite[Chap. 6]{Henkel2014} utilizes these structures to enhance the mixed integer programming based solution algorithm of \cite{HuMitchellPangBennettKunapuli2008} resulting in a significant speed-up. Based on the structural similarities an analogous approach should be possible for risk-averse models after Lagrangean relaxation of coupling constraints.

\medskip

\textbf{Going beyond exogenous stochasticity:} While the analysis in the vast majority of papers on stochastic programming is confined to the case of purely exogenous stochasticity, this assumption is known to be unrealistic in economic models, where the decision maker holds market power. Therefore, models with decision dependent distributions are of particular interest in view of stochastic Stackelberg games (cf. \cite{DeMiguelXu2009}).

\medskip
\small{\textbf{Acknowledgement}
The second author thanks the Deutsche Forschungsgemeinschaft for its support via the Collaborative Research Center TRR 154.}

\end{document}